\documentclass[11pt]{article}
\usepackage{graphicx}
\usepackage{pgfplots}

\usepackage[colorlinks=true,citecolor=blue]{hyperref}
\usepackage{natbib}
\usepackage{graphicx}
\usepackage{amsfonts}
\usepackage{amsmath}
\usepackage{amssymb}
\usepackage{url}
\usepackage{fancyhdr}
\usepackage{indentfirst}
\usepackage{enumerate}
\usepackage{titlesec}
\usepackage{amsthm}
\usepackage{dsfont}
\usepackage[misc]{ifsym}

\usepackage[final]{changes}

\theoremstyle{definition}

\newtheorem{example}{Example}

\newtheorem{definition}{Definition}

\theoremstyle{plain}
\newtheorem{theorem}{Theorem}
\newtheorem{lemma}{Lemma}
\newtheorem{proposition}{Proposition}

\theoremstyle{remark}

\usepackage{cases}

\theoremstyle{definition}

\def\N{\mathbb{N}}

\def\E{\mathbb{E}}

\def\R{\mathbb{R}}

\def\X{\mathcal{X}}

\def\X{\mathcal{X}}

\def\d{\mathrm{d}}

\usepackage[onehalfspacing]{setspace}

\usepackage{bm}
\usepackage{tikz-qtree}
\usepackage{tikz}

 \title{The   reference interval in higher-order stochastic dominance
 }

 \author{Ruodu Wang\thanks{Department of Statistics and Actuarial Science, University of Waterloo,  Canada. \Letter~{\scriptsize\url{wang@uwaterloo.ca}}}
 	\and
 	Qinyu Wu\thanks{Department of Statistics and Actuarial Science, University of Waterloo,  Canada. \Letter~{\scriptsize\url{q35wu@uwaterloo.ca}}}
 }
\date{\today}

\pgfplotsset{compat=1.18}

\begin{document}
\maketitle

\begin{abstract}
Given two random variables taking values in a bounded interval, we study whether one dominates the other in higher-order stochastic dominance depends on the reference interval in the model setting.
We obtain two results. First, the stochastic dominance relations get strictly stronger when the reference interval shrinks if and only if the order of stochastic dominance is larger than three. Second, for mean-preserving stochastic dominance relations,   the reference interval is irrelevant if and only if the difference between the degree of the stochastic dominance and the number of moments is no larger than three.
These results highlight complications arising from using higher-order stochastic dominance in economic applications.

\end{abstract}

\section{Introduction}

Stochastic dominance is a widely used concept in economics, finance, and decision-making under uncertainty, providing a robust method for comparing distributions of uncertain outcomes. This concept is essential in evaluating risk preferences without relying on a specific utility function or preference model, which allows for broad applications across various fields (\cite{L15,P21, SS07}).


First-order stochastic dominance (FSD) and second-order stochastic dominance (SSD) are the most popular stochastic dominance rules. More recently,  the application of higher-order stochastic dominance has become increasingly significant, providing deeper insights into risk behavior that extend beyond mere risk aversion; see \cite{ES06, CET13, DS14,NTK14,LN19}. 

Despite its widespread use, the definition of higher-order stochastic dominance lacks consistency across the literature, sometimes leading to interpretational ambiguity. 
Consider, for example, two distributions, $F$ and $G$, each supported over the interval $[0,1]$. At first glance, one might assume that the question of whether $F$ dominates $G$ in fourth-order stochastic dominance would yield a straightforward answer. However, the consequence can depend significantly on the choice of the reference interval. For instance, if we assess the dominance using only the interval $[0,1]$, $F$ may not dominate $G$. Yet, extending the interval to $[0,2]$ might flip the assessment, resulting in $F$ dominating $G$.\footnote{For a detailed discussion, see Example \ref{ex-counter}, where the specific distributions of $F=(8/9+\epsilon)\delta_{2/9}+(1/9-\epsilon)\delta_1$ and $G=\delta_0/3+2\delta_{4/9}/3$ with $\epsilon=1/100$.
}  This highlights a crucial aspect of higher-order stochastic dominance: It can vary  with alterations in the interval considered. 
This issue has led to   ambiguous formulations of higher-order stochastic dominance across various texts. For example, Definition 7 in \cite{BMM20} and the related definitions in Section 2.3 of \cite{DE10} both adopted an arbitrary interval that encompasses the support of the distribution, but the definition actually depends on the choice of the interval.

To be specific, two prevalent formulations of higher-order stochastic dominance are found in the literature.
The first formulation, denoted as $n{\rm SD}_{\R}$, can be applied to all distributions with bounded support and is defined as: $F$ dominates $G$ if $F^{[n]}(\eta) \leq G^{[n]}(\eta)$ for all $\eta \in \mathbb{R}$, where $F^{[n]}$ is the higher-order cumulative function, as defined in \eqref{eq-intF}; see e.g., \cite{R76,F80,SS07}. The second formulation was initially proposed by \cite{J80} and has been widely adopted in decision theory; see e.g., \cite{EST09, N16, BMM20}. We denote this as $n{\rm SD}_{[a,b]}$, which specifically applies to distributions supported within the interval $[a, b]$. This criterion requires that $F^{[n]}(\eta) \leq G^{[n]}(\eta)$ for all $\eta$ in $[a,b]$, and also the boundary conditions at $b$, i.e., $F^{[k]}(b) \leq G^{[k]}(b)$ for each $k$ from 1 to $n$.
Both formulations can be described by ordering distributions with their expected utility for some sets of utility functions.

We say that the two formulations are {\color{black}consistent} if the ranking between distributions $F$ and $G$ supported in $[a,b]$ remains the same when assessed under $n{\rm SD}_{\R}$ or $n{\rm SD}_{[a,b]}$.
To the best of our knowledge, although various papers  hint at the inconsistency issue under different settings (see the literature review below),
the consistency of the two formulations of higher-order stochastic dominance was explicitly discussed only in \cite{FP22}.
 In their Section 2.2, they contended that $n{\rm SD}_{[a,b]}$ imposes a more stringent criterion than $n{\rm SD}_{\R}$, suggesting that inconsistencies might arise when $n \geq 4$. 
We  formally encapsulate these observations in our Proposition \ref{th-main}, providing a detailed analysis and illustrating the inconsistencies for cases where $n \geq 4$ with a straightforward counterexample in Example \ref{ex-counter}.



Ranking inconsistencies can also arise when applying $n{\rm SD}$ formulations across different intervals, such as $A$ and $B$, when $n \geq 4$. Our Theorem \ref{prop-nSD} illustrates that these inconsistencies arise even when distributions are confined to a subset of the intersection of $A$ and $B$, rather than the entire intersection. This observation highlights the profound influence that the choice of evaluation interval can exert on stochastic dominance assessments, underscoring the importance of meticulous interval selection in both theoretical analysis and practical implementation.

Furthermore, we show that Proposition \ref{th-main} can be extended to a broader class of stochastic dominance rules known as $n$th degree $m$-mean preserving stochastic dominance (\cite{L14}). This framework includes higher-order stochastic dominance, $n$th degree mean-preserving stochastic dominance (\cite{DE13}), and $n$th degree risk increase (\cite{E80}) as special cases. 

One implication of our results is that, since stochastic domination relations (with $n\ge 4$) get strictly stronger when the reference interval shrinks, it affects both their applications and characterization results.
For instance, a stochastic dominance relation is easier to hold when we enlarge the reference interval, which are usually harmless for real-data applications. 
The results also illustrate a drawback of the higher-order stochastic dominance relations. As stochastic dominance is mostly used as a robust tool for ordering risks without assuming specific preferences, 
the fact that they depend on a reference interval --- a subjective choice of the modeler --- jeopardizes their robustness interpretation. 
{\color{black}
From a theoretical standpoint, it is sensible to introduce the reference interval under the assumption that it encompasses the support of all relevant random variables, allowing for a uniform analysis based on higher-order stochastic dominance with respect to this fixed reference interval. 
This is not the case when 
these theoretical results are applied to real-world problems such as portfolio selection and precautionary saving, where the support of the random variables often cannot be objectively specified, making the choice of the reference interval flexible. 
}


\subsection*{\color{black} Literature on  stochastic dominance on grids and sub-intervals}
Many studies have explored stochastic dominance for distributions restricted to specific subsets of $\mathbb{R}$, with several exploring the consistency of these orderings across various subsets. We discuss some of them here. 

\cite{F76} investigated stochastic dominance on a restricted interval $[0,b]$, defined such that $X$ dominates $Y$ if $F_X^{[n]}(\eta) \leq F_Y^{[n]}(\eta)$ for all $\eta$ within $[0,b]$.\footnote{The order $n$ can take any real value from $[1, \infty)$; however, for our purposes, we consider only integer $n$.} Compared to $n{\rm SD}_{[0,b]}$, Fishburn's criterion does not require the boundary condition at $b$, making it less stringent than both $n{\rm SD}_{[0,b]}$ and $n{\rm SD}_{\R}$. Furthermore, \cite{F80} showed that the stochastic dominance relations in \cite{F76} align with $n{\rm SD}_{\R}$ only for $n\in\{1,2\}$.\footnote{\color{black}
For $n\in\{1,2\}$, the boundary condition that $F_X^{[n]}(b)\le F_Y^{[n]}(b)$ guarantees that $F_X^{[n]}(\eta)\le F_Y^{[n]}(\eta)$ for all $\eta> b$. In particular, for  $\eta\ge b$, $F_X^{[1]}(\eta)= F_Y^{[1]}(\eta)=1$, and the comparison of $F_X^{[2]}(\eta)$ and $F_Y^{[2]}(\eta)$ reduces to comparing the expectations of $X$ and $Y$.
However, this is not the case when $n\ge 3$ as higher moments appear. 
}
In light of our Proposition \ref{th-main}, we conclude that $n{\rm SD}_{[0,b]}$, $n{\rm SD}_{\R}$, and Fishburn's criterion exhibit consistent for  $n\in\{1,2\}$. However, for orders $n \geq 3$, these criteria do not exhibit consistency.

The integral stochastic orderings (\cite{W86,M97}) within a fixed subset of $\mathbb{R}$, specifying that $X$ dominants $Y$ if $\E[u(X)] \geq \E[u(Y)]$ for all functions $u$, which is defined on the subset, in a particular class $\mathcal{F}$. Following this framework, \cite{DLU99} and subsequent works by \cite{FL95} and \cite{DL97} explored stochastic orderings by setting the subset as a grid to compare discrete distributions. In particular, 
\cite{DLU99} showed that the ranking of two discrete distributions, each of the support is contained in a fixed grid, can become inconsistent when stochastic orderings are applied to the original grid and then extended to include an additional point. This inconsistency indicates that the choice of the grid significantly affects the ranking of random variables.



\cite{DLS98} and \cite{DVL99} studied $n$-concave orderings within specific intervals, which correspond to increasing $n$th-degree risk as introduced in \cite{E80}. These orderings are always consistent across different intervals (see our Theorem \ref{prop-mainMN}), enabling us to unify their use with the counterpart on $\R$.
Notably, \cite{DLS98} mentioned the possibility of inconsistencies between $n{\rm SD}$ and $n{\rm SD}_{[a,b]}$ in their Remark 3.6, but they did not provide explicit counterexamples or a detailed analysis. Our research builds on these observations and directly addresses these gaps, providing clarification of these potential inconsistencies.

\section{Main results}\label{sec:nsd}
In this section, we will present all the results, while the proofs will be provided in the next section.
For $a,b\in[-\infty,\infty]$ with $a<b$,
denote by $\mathcal X_{[a,b]}$ the set of all bounded random variables taking values in $[a,b]$. For simplicity, we write $\mathcal X:=\mathcal X_{[-\infty,\infty]}$.
We use capital letters, such as $X$ and $Y$,
to represent random variables, and $F$ and $G$ for distribution functions. For $X\in\mathcal X$, we write $\E[X]$ for the expectation of $X$.
Let $F_X$ denote the distribution function of $X$.
We use $\delta_{\eta}$ to represent the point-mass at $\eta\in\R$. For a real-valued function $f$, let $f_-'$ and $f'_+$ be the left and right derivative of $f$, respectively, and denote by $f^{(n)}$ the $n$th derivative for $n \in \mathbb{N}$. Whenever we use the notation $f_-'$, $f'_+$ and $f^{(n)}$, it is understood that they exist.
Denote by $[n]:=\{1,\dots,n\}$ with $n\in\N$.
In this paper, all terms like ``increasing", ``decreasing", ``convex", and ``concave" are in the weak sense.

For a distribution function $F$, denote by $F^{[1]}=F$ and define
\begin{align}\label{eq-intF}
F^{[n]}(\eta)=\int_{-\infty}^\eta F^{[n-1]}(\xi)\d \xi,~~\eta\in\R~{\rm and}~n\ge 2.
\end{align}
It is well-known that $F_X^{[n]}(\eta)$ is connected to the expectation of {\color{black}$(\eta-X)_+^{n-1}$} (see e.g., Proposition 1 of \cite{OR01}):
\begin{align*}
F_X^{[n+1]}(\eta)=\frac{1}{n!}\E[(\eta-X)_+^n],~~X\in\mathcal X,~\eta\in\R,~n\ge 1,
\end{align*}
where $x_+=\max\{0,x\}$ for $x\in\R$. 

As introduced earlier, we now detail the two  formulations of $n$th-order stochastic dominance.


\begin{definition}\label{def:nSD}[{\color{black}\citet[Section 4.A.7]{SS07}}]
Let $n\in\N$. For $X,Y\in\mathcal X$, we say that $X$ dominates $Y$ in the sense of $n$th-order stochastic dominance on $\R$ ($n{\rm SD}_{\R}$), denoted by $X\ge_{n} Y$ or $F_X\ge_{n} F_Y$ if 
\begin{align*}
F_X^{[n]}(\eta)\le F_Y^{[n]}(\eta),~\forall \eta\in\R
\end{align*}
or equivalently,
\begin{align*}
\E[(\eta-X)_+^{n-1}]\le \E[(\eta-Y)_+^{n-1}],~\forall \eta\in\R.
\end{align*}
\end{definition}

\begin{definition}\label{def:nSD[a,b]}[{\color{black}\citet[page 152]{J80}}]
Let $a,b\in[-\infty,\infty]$ with $a<b$ and $n\in\N$. For $X,Y\in\mathcal X_{[a,b]}$, we say that $X$ dominates $Y$ in the sense of $n$th-order stochastic dominance on $[a,b]$ ($n{\rm SD}_{[a,b]}$), denoted by $X\ge_{n}^{[a,b]} Y$ or $F_X\ge_{n}^{[a,b]} F_Y$ if 
\begin{align*}
F_X^{[n]}(\eta)\le F_Y^{[n]}(\eta),~\forall \eta\in[a,b]~{\rm and}~F_X^{[k]}(b)\le F_Y^{[k]}(b)~{\rm for}~ k\in[n]
\end{align*}
or equivalently,
\begin{align*}
\E[(\eta-X)_+^{n-1}]\le \E[(\eta-Y)_+^{n-1}],~\forall \eta\in[a,b]~{\rm and}~\E[(b-X)^{k-1}]\le \E[(b-Y)^{k-1}]~{\rm for}~ k\in[n].
\end{align*}
\end{definition}

In contrast to $n{\rm SD}_{\R}$, $n{\rm SD}_{[a,b]}$ depends on a reference interval and involves additional boundary conditions that $F_X^{[k]}(b)\le F_Y^{[k]}(b)$ of order $k\in[n]$. For $n\in[4]$, 
$n$SD corresponds to the well-known  first, second, third and fourth-order stochastic dominance. Risk aversion, which includes aversion to mean-preserving spreads, aligns with second-order stochastic dominance as described in \cite{RS70}. Higher orders of stochastic dominance, specifically third and fourth orders, cater to decision makers with more refined risk preferences. Third-order stochastic dominance reflects prudence (\cite{K90}), while fourth-order dominance corresponds to temperance (\cite{K92}). The characterizations of these preferences, as detailed by \cite{ES06}, extend from the traditional concept of mean-preserving spreads to a broader framework of risk apportionment.



For two random variables $X,Y\in\mathcal X_{[a,b]}$,  both $n{\rm SD}_{\R}$ and
$n{\rm SD}_{[a,b]}$ can be used to rank their order. This raises a natural question of whether these ranking relations are consistent.  
To examine this, we define two sets of utility functions that 
are regular $n$-increasing concave (\cite{DE13}) over different domains:\footnote{
From the definition in \cite{L14}, the utility functions in $\mathcal U_n$ exhibit $k$th-degree risk aversion for $k \in [n]$. Letting $n \to \infty$, $\mathcal U_{\infty}$ becomes the set of all completely monotone functions, which is well studied in the mathematics literature and is closely linked to Laplace–Stieltjes transforms (see, e.g., \cite{S38}). In this case, utility functions express mixed risk aversion, as discussed in \cite{CP96}.}
\begin{align*}
\mathcal U_{n}^{[a,b]}=\left\{u:\R\to\R \mid \mbox{$(-1)^{k-1}u^{(k)}\ge 0$ on $[a,b]$ for all $k\in[n]$}\right\}
\end{align*}
and
\begin{align*}
\mathcal U_{n}:=\mathcal U_{n}^{[-\infty,\infty]}=\left\{u:\R\to\R \mid \mbox{$(-1)^{k-1}u^{(k)}\ge 0$ on $\R$ for all $k\in[n]$}\right\}.
\end{align*}

Denote $\overline{\mathcal U}_{n}^{[a,b]}$ as the closure of $\mathcal U_{n}^{[a,b]}$ with respect to pointwise convergence. This gives the class of all the utilities such that $(-1)^{k-1}u^{(k)}\ge 0$ for $k\in[n-2]$ and $(-1)^{n-2}u^{(n-2)}$ is increasing and concave on $[a,b]$.
The following proposition provides an answer to the above question of consistency and reveals that the answer is negative for $n\ge 4$.

\begin{proposition}\label{th-main}\label{prop-main}
Let $a,b\in\R$ with $a<b$ and $n\in\N$. For $X,Y\in\mathcal X_{[a,b]}$, we have
\begin{align*}
&\mbox{(i) $X\ge_n^{[a,b]}Y$} \hspace{-0.3cm} & \iff  & \mbox{(ii) $\E[u(X)]\ge \E[u(Y)]$ $\forall u\in\mathcal U_n^{[a,b]}$} \hspace{-0.3cm} & \iff & \mbox{(iii) $\E[u(X)]\ge \E[u(Y)]$ $\forall u\in\overline{\mathcal U}_n^{[a,b]}$}   ; \\  
&\mbox{(iv) $X\ge_nY$} \hspace{-0.3cm}  & \iff  & \mbox{(v) $\E[u(X)]\ge \E[u(Y)]$ $\forall u\in\mathcal U_n$} \hspace{-0.3cm} &\iff  &\mbox{(vi) $\E[u(X)]\ge \E[u(Y)]$ $\forall u\in\overline{\mathcal U}_n$} ;
\end{align*}
and  (i) $\Rightarrow$ (iv) always holds true.  But (iv) $\Rightarrow$ (i) holds for all $X,Y\in \X_{[a,b]}$ if and only if $n\le 3$.
\end{proposition}

The equivalence of (i) and (ii) was well-established; see e.g., \cite{EST09}, \cite{DE13} and Theorem 3.6 of \cite{L15}. Specifically, the implication (i) $\Rightarrow$ (ii) can be shown by using integration by parts. Note that (ii) $\Rightarrow$ (v) is trivial. Once the equivalence of (iv) and (v) is established, it follows that (i) $\Rightarrow$ (iv) holds. This implication suggests that $n{\rm SD}{[a,b]}$ provides a more stringent criterion than $n{\rm SD}_{\R}$ when comparing random variables defined over the space $\mathcal X_{[a,b]}$.
Additionally, Proposition \ref{th-main} demonstrates that a prudent decision maker's preferences remain consistent whether employing $3{\rm SD}_{[a,b]}$ or $3$SD. Consequently, this allows for the uniform application of $3$SD to rank random variables.
However, preferences of a temperate decision maker may vary when transitioning from the criterion $4{\rm SD}$ to $4{\rm SD}_{[a,b]}$. 

{\color{black} A direct explanation for why the implication from (iv) to (i) breaks when the order changes from 3 to 4 is provided below. Note that $n{\rm SD}_{[a,b]}$ involves additional boundary conditions compared to $n{\rm SD}_{\R}$. For $n=3$, suppose that $X\ge_{n}Y$. To check whether $X$ dominates $Y$ under $n{\rm SD}_{[a,b]}$, it is sufficient to verify only one boundary condition: $\E[(b-X)]\le \E[(b-Y)]$, which is equivalent to 
$\E[X]\ge \E[Y]$. This comparison of expectations can be implied from $X\ge_{n}Y$. However, for $n=4$, the boundary condition of order $2$ arises: $\E[(b-X)^2]\le \E[(b-Y)^2]$. In this case, $X\ge_{n}Y$ is not sufficient to establish the boundary condition. 
To illustrate this, 
we present an example involving two random variables $X,Y\in \mathcal{X}_{[0,1]}$ such that $X \geq_4 Y$ and $\E[(1-X)^2]> \E[(1-Y)^2]$, serving as a counterexample to the implication (iv) $\Rightarrow$ (i) in Proposition \ref{th-main} for $n=4$.
}

{\color{black}
\begin{example}\label{ex-counter}
Let $X,Y\in\mathcal X_{[0,1]}$ with 
\begin{align*}
F_X=\left(\frac{8}{9}+\epsilon\right)\delta_{2/9}+\left(\frac{1}{9}-\epsilon\right)\delta_{1}~~{\rm and}~~F_Y=\frac13\delta_0+\frac23\delta_{4/9},~~{\rm where}~\epsilon=\frac{1}{100}.
\end{align*}
Below we consider the rank of $X,Y$ under fourth-order stochastic dominance with different reference intervals. First, we consider $[0,1]$ as the reference interval.
We have the conclusion that $X\not\ge_{4}^{[0,1]}Y$ as the following boundary condition with $k=3$ in Definition \ref{def:nSD[a,b]} is violated:
\begin{align*}
\E\left[\left(1 -X\right)^2\right]-\E\left[\left(1 -Y\right)^2\right]
=\frac{341}{72900}>0.
\end{align*}
Next, the reference interval is set as the entire real line, meaning that we consider fourth-order stochastic dominance as defined in Definition \ref{def:nSD}.
\begin{figure}[h!]
    \centering
    \includegraphics[width=0.6\textwidth]{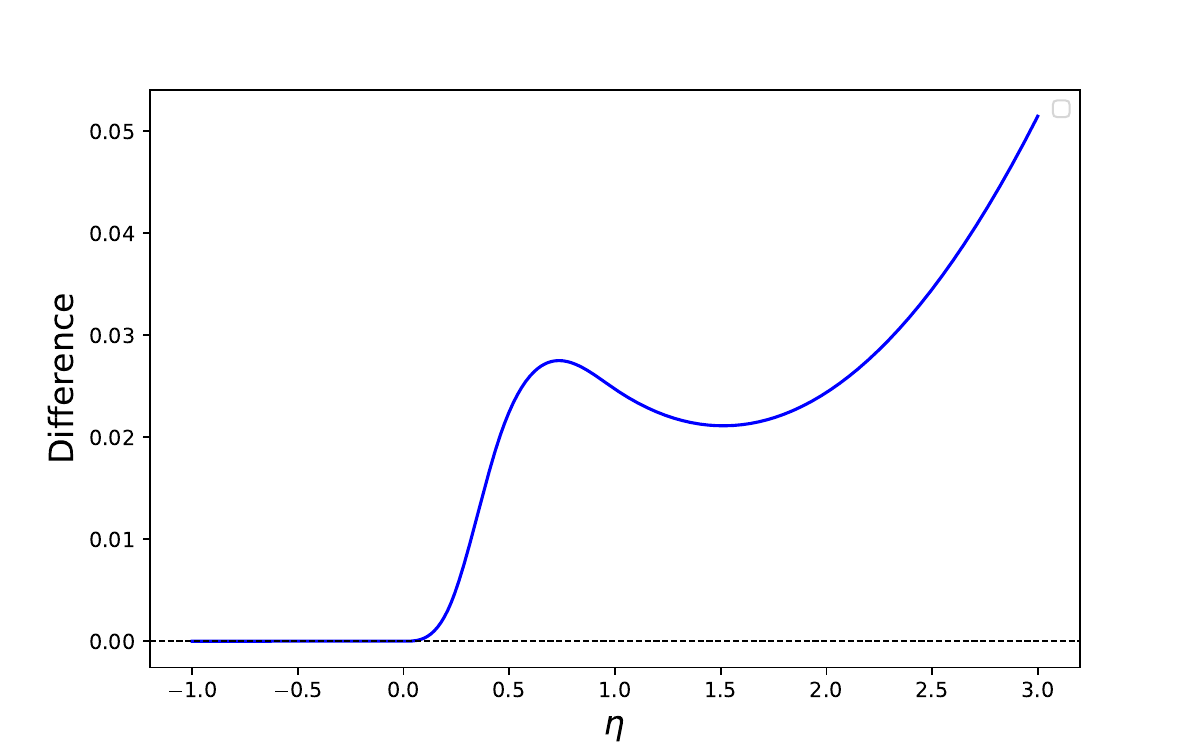} 
    \vspace{-1em} 
    \caption{$\E\big[(\eta - Y)_+^3\big] - \E\big[(\eta - X)_+^3\big]$}
    \label{fig-ex1}
\end{figure}
By standard calculation, one can check that $\E[\left(\eta -X\right)_+^3]\le\E[\left(\eta -Y\right)_+^3]$ for all $\eta\in \R$ (see Figure \ref{fig-ex1} for an intuitive illustration), which implies that {\color{black} $X\ge_{4}Y$}. Finally, we set $[0,2]$ as the reference interval and aim to establish that $X\ge_{4}^{[0,2]}Y$. Note that $\E[\left(\eta -X\right)_+^3]\le\E[\left(\eta -Y\right)_+^3]$ for all $\eta\in \R$. By Definition \ref{def:nSD[a,b]}, it remains to verify the boundary condition with $k\in\{2,3\}$. By standard calculation, we have
\begin{align*}
\E\left[\left(2 -X\right)\right]-\E\left[\left(2 -Y\right)\right]
=-\frac{37}{8100}<0~~{\rm and}~~
\E\left[\left(2 -X\right)^2\right]-\E\left[\left(2 -Y\right)^2\right]
=-\frac{13}{2916}<0.
\end{align*}
Hence, the relation {\color{black} $X\ge_{4}^{[0,2]}Y$} holds.
\end{example}
}




The next result examines the consistency of $n{\rm SD}$ when applied across various intervals. Notably, $n{\rm SD}_{[a,b]}$ does not depend on $a$ as long as $a$ is smaller than the left endpoint of the support of the random variables to compare. For simplicity, we assume that the left endpoint of all intervals is the same in the following theorem.

\begin{theorem}\label{prop-nSD}\label{th-nSD}
Fix $a,b,c,d\in\R$ with $a<b\le c<d$ and $n\in\N$.
\begin{align}
    \label{eq:implies} 
\mbox{For  all $X,Y\in\mathcal X_{[a,b]}$}:  X\ge_n^{[a,c]}Y ~\Longrightarrow~ 
 X\ge_n^{[a,d]}Y.
\end{align} 
The backward implication of \eqref{eq:implies} holds
if and only if  $n\le 3$.
\end{theorem}

Theorem \ref{prop-nSD} illustrates that the ranking of two random variables can be inconsistent when applying $n$SD across different intervals when $n\ge 4$. Specifically, when two stochastic dominance relations, defined over intervals $A$ and $B$ such that the right endpoint of $B$ exceeds that of $A$, this inconsistency arises. Importantly, such inconsistencies occur even when only considering random variables whose support is confined to a sub-interval of $A\cap B$, not necessarily the entire intersection.

In practice, the exact interval that bounds all possible values of wealth may not be known, and decision makers typically set a sufficiently large range based on historical data.   Suppose that there are two risk analysts using 4SD to rank the stock returns in one year. 
One chooses $[ -100\%,2000\%]$ as the reference interval, and one chooses $[ -100\%,1000\%]$ as the reference interval. 
Consider two stock returns, denoted by $X$ and $Y$, evaluated based on their historical performance, both taking values between $[-1,1]$. The first analyst may conclude that $X$ dominates $Y$ in 4SD,
and the second may conclude that the domination does not hold.  
In this example, although both analysts choose very large upper bound $b$ for the interval that surely contains all possible values of $X$ and $Y$, it is unclear which value of $b$ is the right one to choose, and this subjective choice affects their conclusion on domination.
In extreme scenarios, where a new observation shows that the upper bound $b$ is not large enough to cover all risks of interest, the analysts must enlarge their interval, and may arrive at different domination relations even for those return variables that are within the originally chosen interval.


Proposition \ref{prop-main} can be generalized to include a broader category of stochastic dominance rules known as $n$th degree $m$-mean preserving stochastic dominance (\cite{L14}), denoted as $(n,m){\rm SD}_{[a,b]}$. Specifically, for $X, Y \in \mathcal X_{[a,b]}$, $X$ dominates $Y$ in $(n,m){\rm SD}_{[a,b]}$ if $X \ge_n^{[a,b]} Y$ and $F_X^{[k]}(b) \le F_Y^{[k]}(b)$ for all $k \in [n]$, with equality holding for all $k \in [m+1]$. This concept can be extended to the set of all bounded random variables, denoted as $(n,m){\rm SD}_{\R}$, where $X$ dominates $Y$ if $X \ge_n Y$ and $\E[X^k] = \E[Y^k]$ for each $k \in [m]$. The higher-order stochastic dominance is a particular instance of $n$th degree $m$-mean preserving stochastic dominance with $m=0$. If $m=1$, this dominance criterion corresponds to $n$th degree mean-preserving stochastic dominance as introduced in \cite{DE13}. When $m = n-1$, it aligns with the notion of $n$th degree risk increase as originally defined by \cite{E80}.

For the above two formulations of $n$th degree $m$-mean preserving stochastic dominance, we have the following result about the consistency that extends Proposition \ref{prop-main}.

\begin{theorem}\label{prop-mainMN}\label{th-mainMN}
Let $a,b\in\R$ with $a<b$ and $m,n\in\N$ with $m\le n-1$. 
\begin{align}
    \label{eq:implies2} 
\mbox{For  all $X,Y\in\mathcal X_{[a,b]}$}:  \mbox{ $X$ dominates $Y$ in $(n,m){\rm SD}_{[a,b]}$}  ~\Longrightarrow~ 
\mbox{$X$ dominates $Y$ in $(n,m){\rm SD}_{\R}$}.
\end{align} 
 The backward implication of \eqref{eq:implies2}  holds if and only if $n-m\le 3$.
\end{theorem}

Theorem \ref{prop-mainMN} shows that the $n$th degree mean-preserving stochastic dominance rules are consistent for $n \le 4$. Furthermore, the $n$th degree risk increase rules are always consistent across different intervals. This finding indicates that when a decision maker uses the $n$th-degree risk increase rule to compare uncertain outcomes, she can uniformly apply its counterpart on $\R$, making it suitable for all bounded uncertainty outcomes.

\section{Proofs}\label{sec:proof}

\subsection{\color{black} Proof of Proposition \ref{th-main}}
Note that $\overline{\mathcal U}_n^{[a,b]}$ is the closure of ${\mathcal U}_n^{[a,b]}$ with respect to the pointwise convergence. Thus, we only need to consider the statements (i), (ii), (iv) and (v).
We aim to prove Proposition \ref{th-main} by considering three steps in the following.
\begin{enumerate}[(a)]
\item Prove (i) $\Leftrightarrow$ (iv) for $n\le 3$.
\item Prove (i) $\Leftrightarrow$ (ii) $\Rightarrow$ (iv) $\Leftrightarrow$ (v) for $n\in\N$.
\item 
For any $c,d\in(-\infty,\infty)$ with $c<d$ and $n\ge 4$, there exist $X,Y\in \mathcal X_{[c,d]}$ such that $X\ge_{n}Y$ and $X\not\ge_{n}^{[c,d]}Y$.
\end{enumerate}

{\color{black}
Before showing the proof,
we present an auxiliary lemma that will be used in all above steps.
\begin{lemma}[Proposition 6 of \cite{OR01}]\label{lm-OR01}
For $Z\in\mathcal X$ and $n\in\N$, we have
\begin{align*}
\lim_{\eta\to\infty}\left\{\eta-\left(\E[(\eta-Z)_+^n]\right)^{1/n}\right\}=\E[Z].
\end{align*}
As a result, for $X,Y\in\mathcal X$ and $n\in\N$,
$X\ge_n Y$ implies $\E[X]\ge \E[Y]$.
\end{lemma}
}

\begin{proof}[Proof of Step (a)]
The cases $n\in\{1,2\}$ are trivial. Let now $n=3$. 

\underline{(i) $\Rightarrow$ (iv):}
It is straightforward to see that (i) implies $\E[(\eta-X)_+^{2}]\le \E[(\eta-Y)_+^{2}]$ for all $\eta\in(-\infty,b]$ and $\E[X]\ge \E[Y]$. For $\eta>b$, we have
\begin{align*}
\E[(\eta-X)_+^{2}]
&=\E[((b-X)+(\eta-b))^2]
=\E[(b-X)^2]+2(\eta-b)\E[b-X]+(\eta-b)^2\\
&\le \E[(b-Y)^2]+2(\eta-b)\E[b-Y]+(\eta-b)^2
=\E[(\eta-Y)_+^{2}]. 
\end{align*}
This yields (iv).

{\color{black}\underline{(iv) $\Rightarrow$ (i):}
Suppose that $X\ge_{3} Y$. It implies that $\E[(\eta-X)^2_+]\le \E[(\eta-Y)^2_+]$ for all $\eta\in[a,b]$. 
By Lemma \ref{lm-OR01}, we have $\E[X]\ge \E[Y]$, and thus, $\E[(b-X)]\le \E[(b-Y)]$. Hence, we have concluded that $X\ge_{3}^{[a,b]}Y$, and this completes the proof.
}
\end{proof}

\begin{proof}[Proof of Step (b)]
The implication (ii) $\Rightarrow$ (v) is trivial. The implication (v) $\Rightarrow$ (iv) is supported by the fact that the mapping $x\mapsto -(\eta-x)_+^{n-1}$ is contained in $\overline{\mathcal U}_n$, and thus, it can be approximated by a subset of functions in $\mathcal U_n$ with respect to the pointwise convergence.
The implication (i) $\Leftrightarrow$ (ii) has been verified in Theorem 1 of \cite{EST09}. It remains to verify (iv) $\Rightarrow$ (v). In fact,
this implication can be verified by the insight that every $u\in\mathcal U_n$ is a 
positive linear combination of singularity functions in the set $\{f:x\mapsto -(\eta-x)_+^{n-1}|\eta\in\R\}$; see \cite{W56}. 
We provide a self-contained proof  based on the integration by parts below.


Suppose that $X\ge_n Y$.
It follows from Lemma \ref{lm-OR01} that $\E[X]\ge \E[Y]$. 
First, we assume that $\E[X]>\E[Y]$, and the case that $\E[X]=\E[Y]$ will be studied later. Using Lemma \ref{lm-OR01} again and noting that $\E[X]>\E[Y]$, there exists $M\ge b$ such that 
$F_X^{[k]}(M)\le F_Y^{[k]}(M)$ for all $k\in[n-1]$. 
For $u\in\mathcal U_n$, using integration by parts yields
\begin{align*}
\E[u(X)]-\E[u(Y)]& =\int_{\R}u(\eta)\d F_X(\eta)-\int_{\R}u(\eta)\d F_Y(\eta)
\\&=\int_{-\infty}^M (-1)u(\eta)\d \left(F_{Y}(\eta)-F_{X}(\eta)\right)\\
&=(-1)^{-1}u(\eta)(F_{Y}(\eta)-F_{X}(\eta))\big|_{-\infty}^{M}+\int_{-\infty}^M (F_{Y}(\eta)-F_{X}(\eta))\d u(\eta)\\
&=\int_{-\infty}^M(F_{Y}(\eta)-F_{X}(\eta))\d u(\eta) \\& =\int_{-\infty}^M u^{(1)}(\eta)\d \left(F_Y^{[2]}(\eta)-F_X^{[2]}(\eta)\right)\\
&=  u^{(1)}(\eta)\left(F_Y^{[2]}(\eta)-F_X^{[2]}(\eta)\right)\big|_{-\infty}^{M} +\int_{-\infty}^M \left(F_Y^{[2]}(\eta)-F_X^{[2]}(\eta)\right)\d \left( - u^{(1)}(\eta)\right)\\
&\ge \int_{-\infty}^M \left(F_Y^{[2]}(\eta)-F_X^{[2]}(\eta)\right)\d\left( -u^{(1)}(\eta)\right),
\end{align*}
where the inequality follows from $ u^{(1)}\ge 0$  and $F_X^{[2]}(M)\le F_Y^{[2]}(M)$. Using integration by parts repeatedly following a similar argument, we get 
\begin{align*}
& \int_{-\infty}^M \left(F_Y^{[2]}(\eta)-F_X^{[2]}(\eta)\right)\d\left( (-1) u^{(1)}(\eta)\right)
&\ge \int_{-\infty}^M \left(F_Y^{[n]}(\eta)-F_X^{[n]}(\eta)\right)\d \left( (-1)^{n-1} u^{(n-1)}(\eta)\right)\ge 0,
\end{align*} 
where 
the last inequality holds because $F_X^{[n]}(\eta)\le F_Y^{[n]}(\eta)$ for all $\eta\in\R$, and $(-1)^{n-1}u^{(n-1)}$ is increasing as $(-1)^{n-1}u^{(n)}\ge 0$. Hence, we have $\E[u(X)]\ge \E[u(Y)]$ if $\E[X]>\E[Y]$. Suppose now that $\E[X]=\E[Y]$. It is straightforward that $X+\epsilon\ge_n Y$ for all $\epsilon>0$. It follows from the previous arguments that $\E[u(X+\epsilon)]\ge \E[u(Y)]$. 
Note that $\E[u(X+\epsilon)]\to\E[u(X)]$ as $\epsilon\downarrow 0$. This gives $\E[u(X)]\ge \E[u(Y)]$. Hence, we have completed the proof.
\end{proof}

\begin{proof}[Proof of Step (c)]
Let us now focus on Step (c). To verify this step, it suffices to show that there exist $a,b\in(-\infty,\infty)$ with $a<b$ and $X,Y\in\mathcal X_{[a,b]}$ such that $X\ge_{4}Y$ and $X\not\ge_{4}^{[a,b]}Y$, and such example has been given in Example \ref{ex-counter}.
To see this, suppose that $X,Y\in\mathcal X_{[a,b]}$ are the random variables satisfying $X\ge_{4}Y$ and $X\not\ge_{4}^{[a,b]}Y$. Then, we have $\E[(\eta-X)_+^3]\le \E[(\eta-Y)_+^3]$ for all $\eta\in\R$. It follows from Lemma \ref{lm-OR01} that $\E[b-X]\le \E[b-Y]$, and hence, $X\not\ge_{4}^{[a,b]}Y$ implies
$\E[(b-X)^2]>\E[(b-Y)^2]$. For $c,d\in(-\infty,\infty)$ with $c<d$ and $n\ge 4$, define $\widetilde{X}=\lambda X+m$ and $\widetilde{Y}=\lambda Y+m$, where $\lambda=(d-c)/(b-a)$ and $m=(bc-ad)/(b-a)$. It is straightforward to see that $\widetilde{X},\widetilde{Y}\in\mathcal X_{[c,d]}$, and
 $X\ge_4 Y$ implies $\widetilde{X}\ge_n \widetilde{Y}$. Additionally, we have
\begin{align*}
\E[(d-\widetilde{X})^2]=\lambda^2\E[(b-X)^2]>\lambda^2\E[(b-Y)^2]=\E[(d-\widetilde{Y})^2].
\end{align*}
Therefore, we have concluded that $\widetilde{X},\widetilde{Y}\in\mathcal X_{[c,d]}$, $\widetilde{X}\ge_n \widetilde{Y}$ and $\widetilde{X}\not\ge_n^{[c,d]} \widetilde{Y}$, which confirms
Step (c). Hence, we complete the proof. 
\end{proof}

\subsection{Proof of Theorem \ref{prop-nSD}}

We first present an auxiliary lemma.
\begin{lemma}\label{lm-prop1}
Let $a,b\in\R$ with $a<b$. There exists a sequence $\{X_n,Y_n\}_{n\in\N}\subseteq \mathcal X_{[a,b]}^2$ such that $X_n\ge_4 Y_n$ and $\E[X_n]-\E[Y_n]>0$ for all $n\in\N$, and $(\E[X_n^2]-\E[Y_n^2])/(\E[X_n]-\E[Y_n])\to\infty$.
\end{lemma}

\begin{proof}
We assume without loss of generality that $a=0$ and $b=9$. Let $\epsilon_n>0$ and $m_n\in(0,1)$ be such that $\epsilon_n=m_n/(54(1+m_n))$ and $m_n\downarrow 0$. It holds that $m_n/\epsilon_n\downarrow 54$ as $n\to\infty$. Define $\{X_n,Y_n\}_{n\in\N}\subseteq \mathcal X_{[0,9]}^2$ with 
\begin{align*}
F_{X_n}=\left(\frac{8}{9}+\epsilon_n\right)\delta_2+\left(\frac{1}{9}-\epsilon_n\right)\delta_{8+m_n} \mbox{~~and~~}F_{Y_n}=\frac13 \delta_0+\frac23 \delta_{4}.
\end{align*}
It is straightforward to check that $\E[X_n]-\E[Y_n]=5\epsilon_nm_n>0$ for all $n\in\N$.
We aim to verify that $X_n\ge_4 Y_n$ if $n\in\N$ is sufficiently large. To see this, denote by $f_n(\eta)=\E[(\eta-Y_n)_+^3]-\E[(\eta-X_n)_+^3]$ for $n\in\N$ and $\eta\in\R$. It is straightforward to see that $f_n(\eta)\ge 0$ for $\eta\in (-\infty,4]$. By standard calculation, we have
\begin{align*}
9f_n(\eta)=g(\eta)-9\epsilon_n(\eta^3-6\eta^2+12\eta-8),
~~\eta\in[4,8+m_n],
\end{align*}
where $g(\eta)=\eta^3-24\eta^2+192\eta-320$.
One can check that the mapping $g(\eta)$ is increasing on $\R$ and $g(4)>0$. Hence, we have $f_n(\eta)\ge g(4)-9\epsilon_n(\eta^3-6\eta^2+12\eta-8)$ for $\eta\in [4,8+m_n]$.
For sufficiently large $n$, we have $\epsilon_n$ is small enough, and thus, $f_n(\eta)\ge 0$ for $\eta\in [4,8+m_n]$. Let us now consider the case $\eta\ge 8+m_n$.
Denote by $A_n=(1-9\epsilon_n)(6+m_n)$, and it holds that $A_n-6=5m_n^2/(6(1+m_n))>0$ and $A_n\to 6$ as $n\to\infty$. By some standard calculations, we have
\begin{align*}
9f_n(\eta)&=3(A_n-6)\eta^2+3[60-A_n(10+m_n)]\eta+A_n(m_n^2+18m_n+84)-312\\
&\ge -\frac{3[A_n(10+m_n)-60]^2}{4(A_n-6)}+A_n(m_n^2+18m_n+84)-312\\
&=-\frac{3}{4}\left[(10-m_n)\sqrt{A_n-6}-\frac{6m_n}{\sqrt{A_n-6}}\right]^2+A_n(m_n^2+18m_n+84)-312\\
&\to -\frac{3}{4}\lim_{n\to\infty}\frac{36m_n^2}{A_n-6}+200
=-\frac{162}{5}+200>0,
\end{align*}
where we have calculated the minimum of a quadratic function in the  
the first inequality by noting that $A_n>6$, which implies that $f_n(\eta)>0$ for $\eta\ge 8+m_n$ when  $n$ is sufficiently large. Therefore, we have concluded that $X_n\ge_4 Y_n$ if $n\in\N$ is sufficiently large. Note that
\begin{align*}
\frac{\E[X_n^2]-\E[Y_n^2]}{\E[X_n]-\E[Y_n]}&=\frac{-9\epsilon_nm_n^2+m_n^2-144\epsilon_nm_n+16m_n-540\epsilon_n}{m_n-9\epsilon_nm_n-54\epsilon_n}\\
&=\frac{16\frac{m_n}{\epsilon_n}-540-9m_n^2+\frac{m_n^2}{\epsilon_n}-144m_n}{45m_n}\to \infty~~{\rm as}~n\to\infty,
\end{align*}
where the convergence follows from $m_n\downarrow 0$ and $m_n/\epsilon_n\downarrow 54$. This {\color{black}completes} the proof.
\end{proof}

\begin{proof}[Proof of Theorem \ref{th-nSD}]
The implication in \eqref{eq:implies} follows from the equivalence between (i) and (ii) in Proposition \ref{prop-main}. Note that the equivalence between (i) and (iii) in Proposition \ref{prop-main} holds for $n\le 3$, and thus, the backward implication of \eqref{eq:implies} also holds when $n\le 3$. It remains to verify that the backward implication fails if $n\ge 4$. To see this, we assume without loss of generality that $a=0$. Let $X,Y\in\mathcal X_{[0,b]}$ be such that
$\E[X]>\E[Y]$,
\begin{align*}
X\ge_4 Y~~{\rm and}~~ \lambda:=\frac{\E[X^2]-\E[Y^2]}{\E[X]-\E[Y]}\ge 2d,
\end{align*}
where the existence is due to Lemma \ref{lm-prop1}. Let $\gamma\in(2c/\lambda,2d/\lambda]\subseteq (0,1]$, and define $\widetilde{X}=\gamma X$ and $\widetilde{Y}=\gamma Y$. It holds that $\widetilde{X},\widetilde{Y}\in\mathcal X_{[0,b]}$, $\E[\widetilde{X}]>\E[\widetilde{Y}]$,
\begin{align*}
\widetilde{X}\ge_4\widetilde{Y}~~{\rm and}~~ \frac{\E[\widetilde{X}^2]-\E[\widetilde{Y}^2]}{\E[\widetilde{X}]-\E[\widetilde{Y}]}=\gamma\frac{\E[X^2]-\E[Y^2]}{\E[X]-\E[Y]}\in (2c,2d].
\end{align*}
Therefore, we have
\begin{align*}
&\E[(c-\widetilde{X})^2]-\E[(c-\widetilde{Y})^2]
=(\E[\widetilde{X}]-\E[\widetilde{Y}])\left(\frac{\E[\widetilde{X}^2]-\E[\widetilde{Y}^2]}{\E[\widetilde{X}]-\E[\widetilde{Y}]}-2c\right)>0;\\
&\E[(d-\widetilde{X})^2]-\E[(d-\widetilde{Y})^2]
=(\E[\widetilde{X}]-\E[\widetilde{Y}])\left(\frac{\E[\widetilde{X}^2]-\E[\widetilde{Y}^2]}{\E[\widetilde{X}]-\E[\widetilde{Y}]}-2d\right)\le 0.
\end{align*}
This implies that $\widetilde{X}\ge_4^{[0,d]}\widetilde{Y}$ and $\widetilde{X}\not\ge_n^{[0,c]}\widetilde{Y}$ for $n\ge 4$. Since $\ge_4^{[0,d]}$ is more stringent than $\ge_n^{[0,d]}$ for $n\ge 4$, we have completed the proof.
\end{proof}




\subsection{Proof of Theorem \ref{prop-mainMN}}
In this section, closure refers specifically to pointwise convergence.
For $n\ge 3$, define the class of 
$n$-concave functions on $[a,b]$ as follows
\begin{align*}
\mathcal U_{n\text{-cv}}^{[a,b]}=\{u:\R\to\R| \mbox{$(-1)^{n-2}u^{(n-2)}$ is increasing and concave on $[a,b]$}\}.
\end{align*}
The set $\mathcal{U}_{n\text{-cv}}^{[a,b]}$ is a closed convex cone. For any $u\in\mathcal U_{n\text{-cv}}^{[a,b]}$, there exists a sequence $\{u_k\}_{k\in \N}$ such that $(-1)^{n}u_k^{(n)}\le 0$ on $[a,b]$ for all $k\in\N$ and $u_k\to u$ pointwisely.
Denote by $\mathcal U_{(n,m)\text{-cv}}^{[a,b]}=\bigcap_{i=m+1}^n \mathcal U_{i\text{-cv}}^{[a,b]}$, which is also a closed convex cone. The following result is straightforward to verify by sharing a similar proof of Proposition \ref{prop-main} (see also Theorem 1 of \cite{L14}).
\begin{lemma}\label{lm-eqmn}
Let $a,b\in[-\infty,\infty]$ with $a<b$ and $m,n\in\N$ with $m\le n-1$. For $X,Y\in\mathcal X_{[a,b]}$, $X$ dominates $Y$ in $(n,m){\rm SD}_{[a,b]}$ if and only if $\E[u(X)]\ge \E[u(Y)]$ for all $u\in \mathcal U_{(n,m)\text{-\rm cv}}^{[a,b]}$.
\end{lemma}
Note that $\mathcal U_{(n,m)\text{-cv}}^{\R}\subseteq \mathcal U_{(n,m)\text{-cv}}^{[a,b]}$ whenever $a,b\in\R$ and $a<b$. Hence, Lemma \ref{lm-eqmn} yields \eqref{eq:implies2} in Theorem \ref{th-mainMN}. 

Next, we aim to show that the backward implication of \eqref{eq:implies2} holds if $n-m\le 3$. To see this, 
suppose that $X,Y\in\mathcal X_{[a,b]}$ satisfy that $X$ dominates $Y$ in $(n,m){\rm SD}_{\R}$. The cases $n-m\in\{1,2\}$ are trivial. Let now $n-m=3$. It suffices to verify that $\E[(b-X)^{m+1}]\le \E[(b-Y)^{m+1}]$. Since $X$ dominates $Y$ in $(n,m){\rm SD}_{\R}$, we have $X\ge_n Y$. Also note that $\E[X^k]=\E[Y^k]$ for $k\in[m]$. By Theorem 4.A.58 of \cite{SS07}, either $\E[X^{m+1}]=\E[Y^{m+1}]$ or $(-1)^{m+1}\E[X^{m+1}]<(-1)^{m+1}\E[Y^{m+1}]$ holds, which further implies that $\E[(b-X)^{m+1}]\le \E[(b-Y)^{m+1}]$. This completes the proof of the backward implication of \eqref{eq:implies2} for $n-m\le 3$. 

It remains to verify that the backward implication of \eqref{eq:implies2} fails when $n-m \geq 4$. Unlike in Proposition \ref{prop-main}, where a counterexample is presented, here we seek to demonstrate this result through an alternative approach. Such an approach is based on the following lemma, which is a direct result from Corollary 3.8 of \cite{M97}.

\begin{lemma}\label{lm-iff}
Let $a,b\in\R$ with $a<b$, and let $\mathcal V_1$ and $\mathcal V_2$ be two closed convex cones of real-valued functions on $[a,b]$.
Suppose that $\succsim_1$ and $\succsim_2$ are two orderings on $\mathcal X_{[a,b]}$, which satisfy for $X,Y\in\mathcal X_{[a,b]}$ and $i\in\{1,2\}$, $X\succsim_i Y$ if and only if $\E[u(X)]\ge \E[u(Y)]$ for all $u\in\mathcal V_i$. Then, the orderings $\succsim_1$ and $\succsim_2$ are equivalent if and only if $\mathcal V_1=\mathcal V_2$.
\end{lemma}

We complete the proof by 
verifying that $(n,m){\rm SD}_{[a,b]}$ is a strictly more stringent rule than $(n,m){\rm SD}_{[a,c]}$ on $\mathcal X_{[a,b]}$,
where $a,b,c\in\R$ with $a<b<c$ and $m,n\in\N$ with $n-m\ge 4$. 
Choose $u(x)=-(b-x)^{m+2}$ for $x\in\R$. It is straightforward to see that $u\in \mathcal U_{(n,m)\text{-cv}}^{[a,b]}$. Note that $\mathcal U_{(n,m)\text{-cv}}^{[a,b]}$ and $\mathcal U_{(n,m)\text{-cv}}^{[a,c]}$ are both closed convex cone.
Combining Lemmas \ref{lm-eqmn} and \ref{lm-iff}, it suffices to verify that $v\not\equiv u$ on $[a,b]$ for all $v\in \mathcal U_{(n,m)\text{-cv}}^{[a,c]}$. To see this, we assume by contradiction that there exists $v\in \mathcal U_{(n,m)\text{-cv}}^{[a,c]}$ such that $v\equiv u$ on $[a,b]$. It holds that 
\begin{align}\label{eq-v^k}
v^{(k)}(x)=u^{(k)}(x)=\frac{(-1)^{k+1}(m+2)!}{(m+2-k)!}(b-x)^{m+2-k}~~{\rm for}~x\in[a,b]~{\rm and}~k\in[m+2].
\end{align}
Define $f_k(x)=(-1)^{k}v^{(k)}(x)$ for $x\in[a,c]$ and $k\in[m+2]$. We have that
$f_k(x)$ is increasing and concave on $[a,c]$ for $k\in\{m-1,m,m+1,m+2\}$ as $v\in \mathcal U_{(n,m)\text{-cv}}^{[a,c]}$ with $n-m\ge 4$.
Note that $f_{m+2}(x)=-(m+2)!$ on $[a,b]$, and thus, $f_{m+2}$ must be the constant $-(m+2)!$ on $[a,c]$ as it is increasing and concave.
Further, the equation \eqref{eq-v^k} implies $f_{m+1}(x)=-(m+2)!(b-x)$ for $x\in[a,b]$. Since $f_{m+1}^{(1)}(x)=-f_{m+2}=(m+2)!$ for $x\in[a,c]$, we have $f_{m+1}(x)=-(m+2)!(b-x)$ for $x\in[a,c]$. This means that $f_{m+1}(x)>0$ for $x\in[b,c]$. Note that $f_m^{(1)}(x)=-f_{m+1}(x)<0$ for $x\in[b,c]$. This contradicts the fact that $f_m$ is increasing on $[a,c]$. Hence, we have concluded that $(n,m){\rm SD}_{[a,b]}$ is a strictly more stringent rule than $(n,m){\rm SD}_{[a,c]}$ on $\mathcal X_{[a,b]}$, where $a<b<c$ and $n-m\ge 4$. This completes the proof.\qed

\section*{Acknowledgments}

The authors thank Tiantian Mao and Lin Zhao for their helpful discussions and bringing up relevant references. Ruodu Wang is supported by the Natural Sciences and Engineering Research Council of Canada (RGPIN-2018-03823, RGPAS-2018-522590) and Canada Research Chairs (CRC- 2022-00141).

\end{document}